\documentclass[12pt,reqno]{amsart}
\usepackage{amsthm, amscd, amsfonts, amssymb, graphicx, color}
\usepackage[bookmarksnumbered, colorlinks, plainpages,linkcolor=green,anchorcolor=blue,citecolor=blue,urlcolor=blue]{hyperref}
\usepackage{mathrsfs}
\usepackage{amssymb}
\usepackage{enumerate}
\usepackage{exscale}
\usepackage{relsize}
\usepackage{pdfsync} 
\usepackage{tikz}
\tikzset{x=1.3em, y=1.3em}

\makeatletter
\@namedef{subjclassname@2020}{%
	\textup{2020} Mathematics Subject Classification}
\makeatother

\usepackage[includemp,body={398pt,550pt},footskip=30pt,%
marginparwidth=60pt,marginparsep=10pt]{geometry}

\newtheorem{theorem}{Theorem}[section]
\newtheorem{lemma}[theorem]{Lemma}
\newtheorem{proposition}[theorem]{Proposition}
\newtheorem{corollary}[theorem]{Corollary}
\theoremstyle{definition}

\numberwithin{equation}{section}

\newcommand{\D}{\mathcal{D}}
\newcommand{\A}{\mathcal{A}}

\newcommand{\h}{\mathscr{H}}

\newcommand{\R}{\mathcal{R}}

\newcommand{\C}{\mathbb{C}}
\newcommand{\T}{\mathbb{T}}
\newcommand{\TT}{\mathbb{T}^{\infty}}
\newcommand{\DD}{\mathbb{D}}
\newcommand{\E}{\mathbb{E}}

\begin{document}
\title[Littlewood-type theorems for random Dirichlet series]{Littlewood-type theorems for random Dirichlet series}
\date{January 31, 2024.
}

\author[J. Chen, X. Guo and M. Wang]{Jiale Chen, Xin Guo and Maofa Wang}
\address{Jiale Chen, School of Mathematics and Statistics, Shaanxi Normal University, Xi'an 710119, China.}
\email{jialechen@snnu.edu.cn}

\address{Xin Guo, School of Statistics and Mathematics, Zhongnan University of Economics and Law, Wuhan 430073, China}
\email{xguo.math@whu.edu.cn}

\address{Maofa Wang, School of Mathematics and Statistics, Wuhan University, Wuhan 430072, China.}
\email{mfwang.math@whu.edu.cn}

\thanks{Chen was supported by the Fundamental Research Funds for the Central Universities (No. GK202207018) of China. Guo was supported by the NNSF (No. 12101467) of China. Wang was supported by the NNSF (No. 12171373) of China.}



\begin{abstract}
  \noindent
  In this paper, we completely give the solution of the problem of Littlewood-type randomization in the Hardy and Bergman spaces of Dirichlet series. The Littlewood-type theorem for Bergman spaces of Dirichlet series is very different from the corresponding version for Hardy spaces of  Dirichlet series; but also exhibits various pathological phenomena compared with the setting of analytic Bergman spaces over the unit disk, due to the fact that Dirichlet series behave as power series of infinitely many variables. A description for the inclusion between some mixed norm spaces of Dirichlet series plays an essential role in our investigation. Finally, as another application of the inclusion, we completely characterize the superposition operators between Bergman spaces of Dirichlet series.

\end{abstract}
\maketitle
\underline{}

\section{Introduction}
\allowdisplaybreaks[4]

A folklore about the series summation is that a randomized series often enjoys improved regularity.  A well-known example is that the series $\sum_{n=1}^{\infty}\pm n^{-p}$ is almost surely convergent if and only if $p>1/2$.     In the setting of random analytic functions, the classical theorem of Littlewood \cite{Li26,Li30} states that, for any function $f(\xi)=\sum_{n=0}^{\infty}a_n\xi^n$ in the Hardy space $H^2(\DD)$ over the  unit disk $\DD$ of the complex plane $\C$,  the randomization
$$\R f(\xi):=\sum_{n=0}^{\infty}a_nX_n\xi^n$$
belongs to the Hardy space $H^q(\DD)$  alomost surely for all $q>0,$ thus the regularity improves as well.  Here  $\{X_n\}_{n\geq0}$ is a sequence of independent identically distributed Bernoulli  random variables, that is, $\mathbb{P}(X_n=1)=\mathbb{P}(X_n=-1)=1/2$ for all $n\geq 0.$    More precisely, for $p, q\in (0,\infty),$
$$\R:H^p(\DD) \hookrightarrow H^q(\DD) \ \  \text{if and only if} \ \  p\geq 2 \ \ \text{and} \ \   q>0,$$
where $\R:\mathscr{X}\hookrightarrow \mathscr{Y}$ denotes that $\R f\in \mathscr{Y}$ almost surely for any $f\in\mathscr{X}$.

We shall indeed treat three kinds of randomization methods in this note. A random variable  $X$ is called \textit{Bernoulli} if $\mathbb{P}(X=1)=\mathbb{P}(X=-1)=1/2$, \textit{Steinhaus} if it is uniformly distributed on the unit circle, and by $N(0,1),$ we mean the law of \textit{Gaussian} variable with zero mean and unit variance. Moreover, for $ X \in \{\text{Bernoulli, Steinhaus, }N(0,1)\},$ a standard $X$ sequence is a sequence of independent, identically distributed $X$ variables. Lastly, a \textit{standard random sequence} $\{X_{n}  \}_{n\geq0}$  refers to either a standard Bernoulli, Steinhaus or Gaussian $N(0,1)$ sequence.

In contrast to classical Littlewood's theorem (see Fig. 1),  the randomization for functions in the Bergman spaces $L^p_a(\DD)$ over the unit disk $\DD$ exhibits no improvement of regularity for any $p>0$ \cite{CFL22} (see Fig. 2). Actually,   the following Littlewood-type theorem for the Bergman  spaces $L^p_a(\DD)$  was established in \cite[Theorem 2]{CFL22}.

\begin{theorem}[\cite{CFL22}]\label{bergdisk}
Let $0<p,q<\infty$ and $\{ X_{n} \}$ be a standard random sequence. Then $ \R:L^p_a(\DD)\hookrightarrow L^q_a(\DD)  $ if and only if one of the following holds:
\begin{enumerate}
	\item [(i)]      $0<p<2,  \frac{1}{q}-\frac{2}{p}+\frac{1}{2}>0;\ \text{or}$
	\item [(ii)]    $ 2\leq p<\infty, q\leq p. $
\end{enumerate}
\end{theorem}
\vskip.2cm
\begin{tikzpicture}
  \begin{scope}[shift={(0, 12)}]
    \draw (0,0) node[below left=2pt]{$0$} --++(2, 0) node[below=4pt]{$2$} -- ++(5, 0) node[below=4pt]{$\infty$} node[right=6pt]{$p$} -- ++(0, 7) -- ++(-7, 0) node[left=4pt]{$\infty$} node[above=4pt]{$q$} -- cycle;

    \filldraw[fill=gray!20, draw] (2, 0) -- (7, 0) -- (7, 7) -- (2, 7) -- cycle;

    \node[below = 1em] at (3.5, -0.5) {{\bf Fig. 1}};
  \end{scope}
  \begin{scope}[shift={(15, 12)}]
    \draw (0,0) node[below left=2pt]{$0$} --++(2, 0) node[below=4pt]{$2$} -- ++(5, 0) node[below=4pt]{$\infty$} node[right=6pt]{$p$} -- ++(0, 7) -- ++(-7, 0) node[left=4pt]{$\infty$} node[above=4pt]{$q$} -- cycle;

    \fill[gray!20] (0, 0) -- (2, 0) -- (2, 2) -- cycle;
    \draw (0, 0) -- (2, 0);
    \filldraw[fill=gray!20, draw] (2, 0) -- (7, 0) -- (7, 7) -- (2, 2) -- cycle;
    \coordinate (O) at (0, 0);
    \filldraw[fill=white, draw, densely dashed] (O) --plot[domain=0:2, smooth, samples=500](\x, {2*\x / (4 - \x)}) -- cycle;

    \node[below = 1em] at (3.5, -0.5) {{\bf Fig. 2}};
  \end{scope}
\end{tikzpicture}
\vskip.2cm

 Recently,  the analogy version of Littlewood-type theorem has been established in the unit disk setting of Dirichlet spaces \cite{CFL22}, generalized mixed norm spaces \cite{K22} and in the complex plane setting of  Fock spaces   \cite{FT23}.   Note that the mentioned above are function spaces of power series. In this paper, we are going to initiate the investigation of the corresponding phenomenon for random Dirichlet series, which may appear more complexities than power series. Probably the simplest power series is $\sum z^{n}$, whereas the corresponding Dirichler series is $\zeta(s)=\sum n^{-s}$, the Riemann $\zeta$-function, which is automatically interesting and  far less trivial to study. Moreover, there is an important difference: all the radius of convergence coincide for power series, but Dirichlet series has several abscissas of convergence. 

By the end of the 1990s a deep relation between Dirichlet series and different parts of analysis (mainly harmonic and function analysis) was discovered. Ever since, this interaction has shown to be very fruitful, with interesting results in both sides. A new field emerged intertwining the classical work in novel ways with modern  infinite dimensional holomorphy, probability theory as well as analytic number theory.  In this development the  key element is the Hardy spaces of Dirichlet series. Let us  briefly recall the definition of these spaces below. Let $\D$ be the space of Dirichlet series $\sum_{n=1}^{\infty}a_nn^{-s}$  that converge on some half plane $\C_{\theta}:=\{s\in\C:\Re s>\theta\}$, $\theta\in\mathbb{R}$, and let $\mathcal{P}$ be the space of Dirichlet polynomials $\sum_{n=1}^Na_nn^{-s}$. In 1997,  Hedenmalm, Lindqvist and Seip \cite{HLS} introduced the Hardy space $\h^2$ of Dirichlet series and applied it to the Beurling--Wintner problem, which is defined by
$$\h^2:=\left\{f(s)=\sum_{n=1}^{\infty}a_nn^{-s}:
\|f\|_{\h^2}:=\left(\sum_{n=1}^{\infty}|a_n|^2\right)^{1/2}<\infty\right\}.$$
Later, Bayart \cite{Ba} and  Bondarenko et al. \cite{BBSS}  defined Hardy spaces $\h^p$ for general $p>0$. By a basic observation of Bohr, the multiplicative structure of the integers allows us to view an ordinary Dirichlet series of the form $f(s)=\sum_{n=1}^{\infty}a_nn^{-s}$
as a formal power series of infinitely many variables. Indeed, by the transformation $z_{j}=p_{j}^{-s}$ (here $p_{j}$ is the $j$th prime number) and the fundamental theorem of arithmetic, we have the Bohr lift correspondence,
$$f(s)=\sum_{n=1}^{\infty}a_nn^{-s} \ \ \   \longleftrightarrow \ \ \  \mathscr{B}f(z):=\sum_{n=1}^{\infty}a_n z^{\nu(n)}, $$
where $z=(z_1,z_2,\cdots)$, $\nu(n)=(\nu_1(n),\nu_2(n),\cdots)$ is the multi-index such that $n=p_1^{\nu_1(n)}p_2^{\nu_2(n)}\cdots$, and we use the multi-index notation $z^{\nu(n)}=z_1^{\nu_1(n)}z_2^{\nu_2(n)}\cdots$. Then, for $0<p<\infty$, the Hardy space $\h^p$ is defined as the completion of $\mathcal{P}$ with respect to the (quasi-)norm
$$\|P\|_{\h^p}:=\left(\int_{\TT}|\mathscr{B}P(z)|^pdm_{\infty}(z)\right)^{1/p}, \quad P\in\mathcal{P},$$
where $m_{\infty}$ is the Haar measure on the countably infinite-dimensional torus $\TT$ induced by the normalized Lebesgue measure $m$ on the unit circle $\T$. Recall that every element in $\h^p$ is a Dirichlet series that converges uniformly on $\C_{\frac{1}{2}+\epsilon}$ for any $\epsilon>0$, which means that $\h^p $ can be viewed as a space of analytic functions on the half-plane $\C_{1/2}.$ Furthermore, $\h^q$ is contractively embedded into $\h^p$ for $0<p\leq q<\infty$. See \cite{Ba,BBSS,DGMS,HLS,QQ} for more information of these spaces.

Very recently, a study of Bergman spaces of Dirichlet series began in \cite{BL}, which  appears some very different phenomena compared to the Hardy spaces of Dirichlet series. Given $0<p<\infty$ and $\alpha>-1$, the Bergman space $\A^p_{\alpha}$ is defined to be the completion of $\mathcal{P}$ with respect to the (quasi-)norm
$$\|P\|_{\A^p_{\alpha}}:=\left(\int_0^{+\infty}\|P_{\sigma}\|^p_{\h^p}d\mu_{\alpha}(\sigma)\right)^{1/p},\quad P\in\mathcal{P},$$
where $P_{\sigma}(\cdot):=P(\cdot+\sigma)$ and
$$d\mu_{\alpha}(\sigma):=\frac{2^{\alpha+1}}{\Gamma(\alpha+1)}\sigma^{\alpha}e^{-2\sigma}d\sigma.$$
 It is worth mentioning that   $\A^p_{\alpha}$ is a space of Dirichlet series converging on $\C_{1/2}$ as well; see \cite[Theorem 5]{BL}.

 For a Dirichlet series $f(s)=\sum_{n=1}^{\infty}a_nn^{-s}$ in $\D$ and a standard random sequence $\{X_n\}$, the randomization of $f$ is defined by
$$\R f(s):=\sum_{n=1}^{\infty}a_nX_nn^{-s}.$$ As we expected,  the following Littlewood-type theorem for Hardy spaces of Dirichlet series improves the regularity as same as the classical Littlewood theorem (see Fig. 1).

\begin{theorem}\label{remHardy}
Assume  $0<p,q<\infty.$ Let $\{X_n\}$ be a standard random sequence. Then, $\R f\in\h^q$ almost surely for each $f\in\h^p$ if and only if $p\geq2$.
\end{theorem}

In contrast to Theorem \ref{remHardy}, the following Theorem \ref{r-embedding} indicates that the randomization of Dirichlet series in Bergman spaces exhibits no improvement of regularity for any $p>0$ (see Fig. 3-5). Surprisingly,  the loss of regularity for $p<2$ is very drastic and pathological compared with the unit disk setting (see Fig. 2), in the sense that for any $0<q<\infty$ there exists $f\in\A^p_{\alpha}\setminus\A^2_{\alpha}$ such that $\R f\notin\A^q_{\alpha}$ almost surely. The reason may be that, via the Bohr lift, Dirichlet series behave as power series of infinitely many variables.

 \begin{theorem}\label{r-embedding}
	Assume  $0<p,q<\infty$ and $\alpha,\beta>-1.$ Let $\{X_n\}$ be a standard random sequence. Then, $\R f\in\A^q_{\beta}$ almost surely for each $f\in\A^p_{\alpha}$ if and only if one of the following holds:
	\begin{enumerate}
		\item [(i)] $p\geq2$ and $\frac{\alpha+1}{p}<\frac{\beta+1}{q}$;\  or
		\item [(ii)] $q\geq p\geq2$ and $\frac{\alpha+1}{p}=\frac{\beta+1}{q}$.
	\end{enumerate}
\end{theorem}

\vskip.2cm
\begin{center}
\begin{tikzpicture}
	\begin{scope}
		\draw (0,0) node[below left=2pt]{$0$} --++(2, 0) node[below=4pt]{$2$} -- ++(3, 0) node[below=4pt]{$\infty$} node[right=6pt]{$p$} -- ++(0, 6) -- ++(-5, 0) node[left=4pt]{$\infty$} node[above=4pt]{$q$} -- cycle;
		
		\filldraw[fill=gray!20, draw] (2, 0) -- (5, 0) -- (5, 6) -- (2, 2.4) -- cycle;
		
		\node[above left] at (3.8, 3.8) {$q=\frac{\beta + 1}{\alpha + 1}p$};
		
		\node[below = 2em] at (2.5, 0) {$\alpha < \beta$};
		\node[below = 3em] at (2.5, -0.5) {{\bf Fig. 3}};
	\end{scope}
	
	\begin{scope}[shift={(9, 0)}]
		\draw (0,0) node[below left=2pt]{$0$} --++(2, 0) node[below=4pt]{$2$} -- ++(4, 0) node[below=4pt]{$\infty$} node[right=6pt]{$p$} -- ++(0, 6) -- ++(-6, 0) node[left=4pt]{$\infty$} node[above=4pt]{$q$} -- cycle;
		
		\filldraw[fill=gray!20, draw] (2, 0) -- (6, 0) -- (6, 6) -- (2, 2) -- cycle;
		
		\node[above left] at (3.5, 3.1) {$q=p$};

		\node[below = 2em] at (3.5, 0) {$\alpha = \beta$};
		\node[below = 3em] at (3.5, -0.5) {{\bf Fig. 4}};
	\end{scope}
	
	\begin{scope}[shift={(19, 0)}]

		\draw (0,0) node[below left=2pt]{$0$} --++(2, 0) node[below=4pt]{$2$} -- ++(4, 0) node[below=4pt]{$\infty$} node[right=6pt]{$p$} -- ++(0, 5.7) -- ++(-6, 0) node[left=4pt]{$\infty$} node[above=4pt]{$q$} -- cycle;
		
		\fill[gray!20] (2, 0) -- (6, 0) -- (6, 5.7) -- (2, 1.9) -- cycle;
		
		\draw (2, 1.9) -- ++(0, -1.9) -- ++(4, 0) -- ++(0,5.7);
		
		\draw[dashed] (2, 1.9)node[draw, solid, fill=white, circle, inner sep = 2pt]{} -- ++ (4, 3.8);
		
		\node[above left] at (4.2, 3.5) {$q=\frac{\beta + 1}{\alpha + 1}p$};
		
		\node[below = 2em] at (3.5, 0) {$\alpha > \beta$};
		\node[below = 3em] at (3.5, -0.5) {{\bf Fig. 5}};
	\end{scope}
\end{tikzpicture}
\end{center}
\vskip.3cm

The following characterizes the random embedding $\R :  \h^p \hookrightarrow \A^q_{\alpha} $ for $0<p,q<\infty$ and  $\alpha>-1$, which shows no improvement of Theorem   \ref{remHardy} and is quite different from  the case of power series (see \cite[Theorem 19]{CFL22}).  For curious readers, the answer to $\R:\A^p_{\alpha}\hookrightarrow\h^q$ is trivial due to  Proposition \ref{hardy}.

\begin{theorem}\label{Rmix}
 Assume $0<p,q<\infty$ and  $\alpha>-1.$ Let $\{X_n\}$ be a standard random sequence.    Then, $\R f\in \A^q_{\alpha}  $ almost surely for each $f\in \h^p $ if and only if $p\geq2$.
\end{theorem}

In order to establish the above Littlewood-type theorems for random Dirichlet series, we need to determine when a random Dirichlet series almost surely belongs to the Hardy spaces $\h^p$ or the Bergman spaces $\A^p_{\alpha}$. This very active topic for random analytic functions has been investigated in the unit disk setting such as Hardy spaces \cite{K85,Li26,Li30}, the $H^{\infty}$ space \cite{B63, K85, MP78, MP81, SZ54}, the Bloch space \cite{ACP74, G}, BMOA \cite{D85,NP23, S81}, Dirichlet spaces \cite{CFL22,CSU} and Bergman spaces \cite{CFL22,K22}; the complex plane  setting of Fock spaces \cite{FT23}; and the Dirichlet series setting of BMOA \cite{KQSS} recently. The study of  random Dirichlet series can be traced back to \cite{CDS,CMST,H39,K85}, especially the  natural boundary problem \cite{BM} and the distribution of real zeros \cite{A19,AFM}. In the viewpoint of Banach spaces, we cannot  find any references devoted to random Dirichlet series of Bergman spaces. Generally, we here introduce the  mixed norm space $\h^{p,q}_{\alpha}$ of Dirichlet series, which is the natural generalization of the Bergman space $\A^p_{\alpha}$ and plays a prominent role in our investigation. Given $0<p,q<\infty$ and $\alpha>-1$, the mixed norm space $\h^{p,q}_{\alpha}$ is defined as the completion of $\mathcal{P}$ with respect to the (quasi-)norm
$$\|P\|_{\h^{p,q}_{\alpha}}:=\left(\int_0^{+\infty}\|P_{\sigma}\|^q_{\h^p}d\mu_{\alpha}(\sigma)\right)^{1/q}, \quad P\in\mathcal{P}.$$
We also write $\h^{p,\infty}_{\alpha}:=\h^p$ for convenience. It is clear that $\h^{p,p}_{\alpha}=\A^p_{\alpha}$. Moreover, $\h^{p,q}_{\alpha}$ is a Banach space if $p,q\geq1$; otherwise, it is an $r$-Banach space with $r=\min\{p,q\}$.

  Based on the Marcinkiewicz--Zygmund--Kahane theorem, the Fernique theorem and the Khintchine--Kahane inequality for $r$-Banach spaces, we have the following result.

\begin{theorem}\label{symbolA}
Assume $0<p<\infty$ and $\alpha>-1.$ Let $\{X_n\}$ be a standard random sequence. Suppose that $f\in\D$. Then
\begin{enumerate}
	\item [(i)] $\R f\in\h^p$ almost surely if and only if $f\in\h^2$;
	\item [(ii)] $\R f\in\A^p_{\alpha}$ almost surely if and only if $f\in\h^{2,p}_{\alpha}$.
\end{enumerate}
\end{theorem}

Let $\mathscr{X}\subset\D$ be a (quasi-)Banach space of Dirichlet series containing all Dirichlet polynomials. Then by the Hewitt--Savage zero-one law (see \cite[Theorem 2.5.4]{Du19}), for any $f\in\D$,
$$\mathbb{P}(\R f\in\mathscr{X})\in\{0,1\}.$$
This prompts  us to define the symbol space $\mathscr{X}_{\star}$ by
$$\mathscr{X}_{\star}:=\left\{f\in\D:\mathbb{P}(\R f\in\mathscr{X})=1\right\}.$$
Then Theorem \ref{symbolA} can be rewritten as  $(\A^p_{\alpha})_{\star}=\h^{2,p}_{\alpha}$ and  $(\h^p)_{\star}=\h^2$,  which not only extends \cite[Proposition 4]{CDS} for all $0<p<\infty$ and standard random sequences, but also establishes the Bergman space analogy.  In fact, we completely solve the corresponding problem for mixed norm spaces of Dirichlet series; see  Corollary \ref{sym} for the characterization of $(\h^{p,q}_{\alpha})_{\star}$  for $0<p<\infty$, $0<q\leq\infty$ and $\alpha>-1,$ which reduces to Theorem \ref{symbolA} by choosing $q=\infty$ and $p=q$  respectively.

Based on Theorem \ref{symbolA} (i), the random embedding $\R:\h^p\hookrightarrow\h^q$ is equivalent to the inclusion $\h^p\subset(\h^q)_{\star}=\h^2$. Hence Theorem \ref{remHardy} holds automatically if we establish Theorem \ref{symbolA} (i). Similarly, Theorem \ref{symbolA} (ii) implies that the random embedding $\R:\A^p_{\alpha}\hookrightarrow \A^q_{\beta}$ is equivalent to the inclusion $\A^p_{\alpha}\subset(\A^q_{\beta})_{\star}=\h^{2,q}_{\beta}$. Hence to establish Theorem \ref{r-embedding}, we need to characterize the indices $p,q,\alpha,\beta$ such that the inclusion $\A^p_{\alpha}\subset\h^{2,q}_{\beta}$ holds. Compared with \cite{CFL22}, the main challenge here is to circumvent the difficulty caused by coefficient multipliers which  work effectively in the unit disk setting but are less known in the Dirichlet series setting. We apply different methods to completely characterize the inclusion between mixed norm spaces of Dirichlet series with six parameters, which is of independent interest.

\begin{theorem}\label{em}
Let $0<p,q,u,v<\infty$ and $\alpha,\beta>-1$. Then $\h^{p,q}_{\alpha}\subset\h^{u,v}_{\beta}$ if and only if
\begin{enumerate}
	\item [(i)] $p\geq u$ and $\frac{\alpha+1}{q}<\frac{\beta+1}{v}$; or
	\item [(ii)] $p\geq u$, $\frac{\alpha+1}{q}=\frac{\beta+1}{v}$ and $q\leq v$.
\end{enumerate}
\end{theorem}

Finally, we give another application of the above theorem. Given a function $\varphi$ on the complex plane $\C$, the nonlinear superposition operator $S_{\varphi}$ is defined for complex-valued functions $f$ by
$$S_{\varphi}f=\varphi\circ f.$$
Given two function spaces $\mathscr{X}$ and $\mathscr{Y}$, a natural question is to find the functions $\varphi$ such that the operators $S_{\varphi}$ map $\mathscr{X}$ into $\mathscr{Y}$. The theory of superposition operators has a long history in the context of real-valued functions; see \cite{AZ,B,BK,BLS0,BLS,MM}. Recently, the superposition operators on spaces of analytic functions have drawn many attentions; see \cite{Ca95,CaG94,DoGi,Me22,SW} and the references therein. However, there have been just a few studies in the setting of Dirichlet series. Bayart et al. \cite[Theorem 9]{BCGMS} proved that a function $\varphi$ induces a superposition operator $S_{\varphi}$ mapping $\h^p$ into $\h^q$ if and only if $\varphi$ is a polynomial of degree at most $p/q$. Fern\'{a}ndez Vidal et al. \cite{FGMS} investigated the superposition operators on some Fr\'{e}chet spaces of Dirichlet series. Different from the  methods in \cite{BCGMS,CaG94,FGMS}, we apply Theorem \ref{em} fortunately to characterize the superposition operators between Bergman spaces of Dirichlet series below. In fact, we completely characterize the superposition operators $S_{\varphi}$ between  the mixed norm spaces $\h^{p,q}_{\alpha}$ and $\h^{u,v}_{\beta}$ for all $0<p,q,u,v<\infty$ and $\alpha,\beta>-1$; see Theorem \ref{supe}.

\begin{theorem}\label{sBergman}
Let $0<p,q<\infty$, $\alpha,\beta>-1$, and let $\varphi$ be a function on $\C$. Then the superposition operator $S_{\varphi}$ maps $\A^p_{\alpha}$ into $\A^q_{\beta}$ if and only if $\varphi$ is a polynomial of degree $N$, where $N$ satisfies
\begin{enumerate}
	\item [(i)] $N\leq\frac{p}{q}$ in the case $\alpha\leq\beta$;
	\item [(ii)] $N<\frac{p(\beta+1)}{q(\alpha+1)}$ in the case $\alpha>\beta$.
\end{enumerate}
Moreover, if $S_{\varphi}$ maps $\A^p_{\alpha}$ into $\A^q_{\beta}$, then it is bounded and continuous.
\end{theorem}

 The paper is organized as follows. In Section \ref{sec2}, we establish some fundamental function-theoretic  properties for the mixed norm spaces $\h^{p,q}_{\alpha}$. In Section \ref{sec3}, we establish  the  equivalent descriptions for a random Dirichlet series to be in the mixed norm spaces  $\h^{p,q}_{\alpha}$ almost surely, which covers Theorem \ref{symbolA}. Section \ref{sec4} is devoted to proving Theorem \ref{em}, which leads to a proof of Theorem \ref{r-embedding}. Finally, we  investigate  the superposition operators between Bergman spaces of Dirichlet series in Section \ref{super}.

Throughout the paper, we use $A\lesssim B$ or $B\gtrsim A$ to denote that $A\leq CB$ for some inessential constant $C>0$. If $A\lesssim B\lesssim A$, then we write $A\asymp B$. For $\sigma\geq0$ and $f\in\D$, write $f_{\sigma}(s)=T_{\sigma}f(s):=f(s+\sigma)$. We assume that all random variables are defined on a probability space $(\Omega,\Sigma,\mathbb{P})$ with expectation denoted by $\mathbb{E}$.

\section{Mixed norm spaces of Dirichlet series}\label{sec2}

In this section, we establish some fundamental properties for the mixed norm spaces $\h^{p,q}_{\alpha}$. The following lemma  concerns the point evaluations on $\h^{p,q}_{\alpha}$.

\begin{lemma}\label{growth}
Let $0<p,q<\infty$ and $\alpha>-1$. Then the point evaluation at any $s\in\C_{1/2}$ extends to a bounded functional on $\h^{p,q}_{\alpha}$. Moreover, for $f\in\h^{p,q}_{\alpha}$,
\begin{equation}\label{gr}
|f(s)|\lesssim\left(\frac{\Re s}{2\Re s-1}\right)^{\frac{1}{p}+\frac{\alpha+1}{q}}\|f\|_{\h^{p,q}_{\alpha}},\quad s\in\C_{1/2}.
\end{equation}
\end{lemma}
\begin{proof}
It is sufficient to establish \eqref{gr} for any Dirichlet polynomial $P$. Fix $s\in\C_{1/2}$. Then for any $0<\sigma<(\Re s-1/2)/2$, it follows from \cite[Theorem 3]{Ba} that
$$|P(s)|^p=|P_{\sigma}(s-\sigma)|^p\leq\zeta(2\Re s-2\sigma)\|P_{\sigma}\|^p_{\h^p},$$
which implies that
$$|P(s)|^q\int_0^{\frac{\Re s-1/2}{2}}d\mu_{\alpha}(\sigma)\leq
\int_0^{\frac{\Re s-1/2}{2}}\zeta(2\Re s-2\sigma)^{q/p}\|P_{\sigma}\|^q_{\h^p}d\mu_{\alpha}(\sigma).$$
Noting that $\zeta(x)\leq\frac{x}{x-1}$ for $x>1$, we obtain that for $0<\sigma<(\Re s-1/2)/2$,
$$\zeta(2\Re s-2\sigma)\leq\zeta(\Re s+1/2)\leq\frac{2\Re s+1}{2\Re s-1}.$$
Therefore,
\begin{equation}\label{inte}
|P(s)|^q\int_0^{\frac{\Re s-1/2}{2}}d\mu_{\alpha}(\sigma)\leq
\left(\frac{2\Re s+1}{2\Re s-1}\right)^{q/p}\|P\|^q_{\h^{p,q}_{\alpha}}.
\end{equation}
If $(\Re s-1/2)/2\geq1$, i.e. $\Re s\geq5/2$, then $\int_0^{\frac{\Re s-1/2}{2}}d\mu_{\alpha}(\sigma)\gtrsim1$. Consequently,
$$|P(s)|^q\lesssim\left(\frac{2\Re s+1}{2\Re s-1}\right)^{q/p}\|P\|^q_{\h^{p,q}_{\alpha}}
\lesssim\|P\|^q_{\h^{p,q}_{\alpha}}\lesssim\left(\frac{\Re s}{2\Re s-1}\right)^{\frac{q}{p}+\alpha+1}\|P\|^q_{\h^{p,q}_{\alpha}}.$$
If $(\Re s-1/2)/2<1$, i.e. $\Re s<5/2$, then
$$\int_0^{\frac{\Re s-1/2}{2}}d\mu_{\alpha}(\sigma)
=\frac{2^{\alpha+1}}{\Gamma(\alpha+1)}\int_0^{\frac{\Re s-1/2}{2}}\sigma^{\alpha}e^{-2\sigma}d\sigma
\asymp(2\Re s-1)^{\alpha+1},$$
which, together with \eqref{inte}, implies that
$$|P(s)|^q\lesssim\frac{(2\Re s+1)^{q/p}}{(2\Re s-1)^{q/p+\alpha+1}}\|P\|^q_{\h^{p,q}_{\alpha}}
\lesssim\left(\frac{\Re s}{2\Re s-1}\right)^{\frac{q}{p}+\alpha+1}\|P\|^q_{\h^{p,q}_{\alpha}}.$$
Hence \eqref{gr} holds for any Dirichlet polynomial and the proof is complete.
\end{proof}

We will need the following lemma, which was proved for $p\geq1$ in \cite[Proposition 2.3]{DP} and for $p<1$ in \cite[Proposition 2.5]{FGY}.

\begin{lemma}\label{decreasing}
Given $0<p<\infty$ and $f\in\h^p$, $f_{\sigma}$ belongs to $\h^p$ for any $\sigma>0$. Moreover, the function $\sigma\mapsto\|f_{\sigma}\|_{\h^p}$ is decreasing on $[0,+\infty)$.
\end{lemma}

The next theorem indicates that every element in the mixed norm space $\h^{p,q}_{\alpha}$ is a Dirichlet series.

\begin{theorem}\label{sigma}
Let $0<p,q<\infty$ and $\alpha>-1$. If $f\in\h^{p,q}_{\alpha}$, then $f_{\sigma}\in\h^p$ for every $\sigma>0$, and
\begin{equation}\label{norm}
\|f\|_{\h^{p,q}_{\alpha}}=\left(\int_0^{+\infty}\|f_{\sigma}\|^q_{\h^p}d\mu_{\alpha}(\sigma)\right)^{1/q}.
\end{equation}
In particular, every element in $\h^{p,q}_{\alpha}$ is a Dirichlet series that converges uniformly on $\C_{\frac{1}{2}+\epsilon}$ \  for any $\epsilon>0$.
\end{theorem}
\begin{proof}
Let $\{P_j\}_{j\geq1}$ be a sequence of Dirichlet polynomials that converges to $f$ in $\h^{p,q}_{\alpha}$. Fix $\sigma>0$. Then for any $m,j\geq1$, Lemma \ref{decreasing} yields that
$$\|P_m-P_j\|^q_{\h^{p,q}_{\alpha}}\geq\int_0^{\sigma}\|(P_m-P_j)_{\delta}\|^q_{\h^p}d\mu_{\alpha}(\delta)
\geq\|(P_m)_{\sigma}-(P_j)_{\sigma}\|^q_{\h^p}\int_{0}^{\sigma}d\mu_{\alpha}(\delta),$$
which implies that $\{(P_j)_{\sigma}\}_{j\geq1}$ is a Cauchy sequence in $\h^p$. Consequently, there exists $g\in\h^p$ such that $(P_j)_{\sigma}\to g$ in $\h^p$ as $j\to\infty$. Then for any $s\in\C_{1/2}$, $g(s)=\lim_{j\to\infty}(P_j)_{\sigma}(s)$. Moreover, combining the fact that $P_j\to f$ in $\h^{p,q}_{\alpha}$ with Lemma \ref{growth} gives that
$$f_{\sigma}(s)=f(\sigma+s)=\lim_{j\to\infty}P_j(\sigma+s)=\lim_{j\to\infty}(P_j)_{\sigma}(s),\quad s\in\C_{1/2}.$$
Therefore, $f_{\sigma}=g\in\h^p$. Since $\sigma>0$ is arbitrary, we obtain that $f$ is a Dirichlet series that converges uniformly on $\C_{\frac{1}{2}+\epsilon}$ for any $\epsilon>0$.  In terms of the proof in  \cite[Theorem 3.5]{FGY}, we can deduce
$$\lim_{j\to\infty}\int_0^{+\infty}\|f_{\sigma}-(P_j)_{\sigma}\|^q_{\h^p}d\mu_{\alpha}(\sigma)=0,$$
which, together with the triangle (or H\"{o}lder) inequality, implies \eqref{norm}. We omit the details.
\end{proof}

Recall that for any $0<p<\infty$, the Dirichlet polynomials are dense in the Hardy space $\h^p$. Hence for any $f\in\h^p$,
\begin{equation}\label{converge}
\lim_{\sigma\to0^+}\|f_{\sigma}-f\|_{\h^p}=0.
\end{equation}
We now consider the action of the horizontal translation on $\h^{p,q}_{\alpha}$.

\begin{lemma}\label{con-sigma}
Let $0<p,q<\infty$ and $\alpha>-1$. Then the following assertions holds:
\begin{enumerate}
	\item[(1)] $\h^p\subset\h^{p,q}_{\alpha}$, and for any $f\in\h^p$, $\|f\|_{\h^{p,q}_{\alpha}}\leq\|f\|_{\h^p}$;
	\item[(2)]   for any $\sigma>0$, $T_{\sigma}$ is a contraction on $\h^{p,q}_{\alpha}$;
	\item[(3)]   for any $f\in\h^{p,q}_{\alpha}$, $f_{\sigma}\to f$ in $\h^{p,q}_{\alpha}$ as $\sigma\to0^+$.
\end{enumerate}
\end{lemma}
\begin{proof}
(1) This is clear in view of Lemma \ref{decreasing} and the density of Dirichlet polynomials in $\h^p$.

(2) By Lemma \ref{decreasing}, we know that $\|T_{\sigma}P\|_{\h^{p,q}_{\alpha}}\leq\|P\|_{\h^{p,q}_{\alpha}}$ for any Dirichlet polynomial $P$. Consequently, $T_{\sigma}$ can be extended to a contraction on $\h^{p,q}_{\alpha}$.

(3) Applying Theorem \ref{sigma} and the dominated convergence theorem, we obtain that
$$\|f-f_{\sigma}\|^q_{\h^{p,q}_{\alpha}}=\int_{0}^{+\infty}\|f_{\delta}-f_{\sigma+\delta}\|^q_{\h^p}d\mu_{\alpha}(\delta)\to0$$
as $\sigma\to0^+$.
\end{proof}

Given $0<p,q<\infty$ and $\alpha>-1$, let $L^q(d\mu_{\alpha};\h^p)$ be the space of $\h^p$-valued functions $F$ on $(0,+\infty)$ such that
$$\|F\|^q_{L^q(d\mu_{\alpha};\h^p)}:=\int_0^{+\infty}\|F(\sigma)\|^q_{\h^p}d\mu_{\alpha}(\sigma)<\infty.$$
For any $\delta>0$, let $\widetilde{T}_{\delta}$ be defined for $F\in L^q(d\mu_{\alpha};\h^p)$ by
$$\big(\widetilde{T}_{\delta}F\big)(\sigma):=T_{\delta}\big(F(\sigma)\big),\quad \sigma\in(0,+\infty).$$
Then it follows from Lemma   \ref{decreasing}    that $\widetilde{T}_{\delta}$ is a contraction on $L^q(d\mu_{\alpha};\h^p)$ for each $\delta>0$. Moreover, for any $F\in L^q(d\mu_{\alpha};\h^p)$, the dominated convergence theorem yields that
\begin{equation}\label{lq0}
\lim_{\delta\to0^+}\left\|\widetilde{T}_{\delta}F-F\right\|^q_{L^q(d\mu_{\alpha};\h^p)}
=\lim_{\delta\to0^+}\int_0^{+\infty}\left\|T_{\delta}\big(F(\sigma)\big)-F(\sigma)\right\|^q_{\h^p}d\mu_{\alpha}(\sigma)=0.
\end{equation}

\begin{theorem}\label{suff}
Let $0<p,q<\infty$ and $\alpha>-1$. Suppose that $f\in\D$ satisfies that $f_{\sigma}\in\h^p$ for all $\sigma>0$ and
$$\int_0^{+\infty}\|f_{\sigma}\|^q_{\h^p}d\mu_{\alpha}(\sigma)<\infty.$$
Then $f\in \h^{p,q}_{\alpha}$.
\end{theorem}
\begin{proof}
Write $G(\sigma)=f_{\sigma}$ for $\sigma>0$. Then $G\in L^q(d\mu_{\alpha};\h^p)$. Fix $\epsilon>0$. In view of  \eqref{lq0}, there exists $\delta_0>0$ such that for any $0<\delta<\delta_0$,
$$\left\|\widetilde{T}_{\delta}G-G\right\|_{L^q(d\mu_{\alpha};\h^p)}<\epsilon.$$
Then for any $0<\sigma_1<\sigma_2<\delta_0$, bearing in mind that $T_{\sigma_1}$ is a contraction on $\h^p$, we deduce from Theorem \ref{sigma} and Lemma \ref{con-sigma} that
\begin{align*}
\|f_{\sigma_1}-f_{\sigma_2}\|^q_{\h^{p,q}_{\alpha}}
&=\int_0^{+\infty}\left\|(f_{\sigma_1})_{\sigma}-(f_{\sigma_2})_{\sigma}\right\|^q_{\h^p}d\mu_{\alpha}(\sigma)\\
&=\int_0^{+\infty}\left\|T_{\sigma_1}(f_{\sigma})-T_{\sigma_2}(f_{\sigma})\right\|^q_{\h^p}d\mu_{\alpha}(\sigma)\\
&\leq\int_0^{+\infty}\left\|f_{\sigma}-T_{\sigma_2-\sigma_1}(f_{\sigma})\right\|^q_{\h^p}d\mu_{\alpha}(\sigma)\\
&=\int_0^{+\infty}\left\|G(\sigma)-T_{\sigma_2-\sigma_1}\big(G(\sigma)\big)\right\|^q_{\h^p}d\mu_{\alpha}(\sigma)\\
&=\left\|G-\widetilde{T}_{\sigma_2-\sigma_1}G\right\|^q_{L^q(d\mu_{\alpha};\h^p)}\\
&<\epsilon^q.
\end{align*}
Consequently, there exists $g\in\h^{p,q}_{\alpha}$ such that $f_{\sigma}\to g$ in $\h^{p,q}_{\alpha}$ as $\sigma\to0^+$. Then Lemma \ref{growth} implies that $g(s)=\lim_{\sigma\to0^+}f_{\sigma}(s)$ for any $s\in\C_{1/2}$. Combining this with the fact that $f_{\sigma}\in\h^p$ for all $\sigma>0$, we have that
$$f(s)=\lim_{\sigma\to0^+}f(s+\sigma)=\lim_{\sigma\to0^+}f_{\sigma}(s)=g(s),\quad \forall s\in\C_{1/2}.$$
Therefore, $f=g\in\h^{p,q}_{\alpha}$.
\end{proof}

Note that for $f(s)=\sum_{n=1}^{\infty}a_nn^{-s}$ in $\D$ and $\sigma>0$, $\|f_{\sigma}\|^2_{\h^2}=\sum_{n=1}^{\infty}|a_n|^2n^{-2\sigma}$. As an immediate consequence of Theorems \ref{sigma} and \ref{suff}, we have
\begin{equation}\label{222}
f\in\h^{2,q}_{\alpha}\quad  \Longleftrightarrow\quad \int_{0}^{+\infty}\left(\sum_{n=1}^{\infty}|a_n|^2n^{-2\sigma}\right)^{q/2}d\mu_{\alpha}(\sigma)<\infty.
\end{equation}

\section{The symbol spaces $(\h^{p,q}_{\alpha})_{\star}$}\label{sec3}

The purpose of this section is to determine when a random Dirichlet series almost surely belongs to the Hardy spaces $\h^p$ or the Bergman spaces $\A^p_{\alpha}$. In fact, we will completely describe the symbol spaces $(\h^{p,q}_{\alpha})_{\star}$ for all $0<p<\infty$, $0<q\leq\infty$ and $\alpha>-1$.

For every positive integer $N$, let $S_N$ be the partial sum operator defined by
$$S_N\left(\sum_{n=1}^{\infty}a_nn^{-s}\right)=\sum_{n=1}^Na_nn^{-s}.$$
Recall that if $p>1$, then the operators $S_N$ are uniformly bounded on $\h^p$, i.e.
$$\sup_{N\geq1}\|S_N\|_{\h^p\to\h^p}<\infty$$
(see \cite[Corollary 4]{AOS}), which easily implies that $S_N$ are uniformly bounded on $\h^{p,q}_{\alpha}$. Consequently, in the case $p>1$ and $q\geq1$, the set $\{n^{-s}:n\geq1\}$ forms a Schauder basis of $\h^{p,q}_{\alpha}$, and $S_Nf\to f$ in $\h^{p,q}_{\alpha}$ for any $f\in\h^{p,q}_{\alpha}$. The following lemma indicates that, when we consider random Dirichlet series, then the above convergence of partial sum holds for all $p,q>0$. This is quite interesting since it was proved in \cite[Section 5]{BBSS} that for $p\leq1$,
$\|S_N\|_{\h^p\to\h^p}\to+\infty$ as $N\to\infty$.

\begin{lemma}\label{random-partial}
 Assume $0<p<\infty$, $0<q\leq\infty$ and  $\alpha>-1.$ Let $\{X_n\}$ be a sequence of independent and symmetric random variables. If $f(s)=\sum_{n=1}^{\infty}a_nX_nn^{-s}$ belongs to $\h^{p,q}_{\alpha}$ almost surely, then $S_Nf(s)=\sum_{n=1}^{N}a_nX_nn^{-s}$ converges to $f$ in $\h^{p,q}_{\alpha}$ almost surely as $N\to\infty$.
\end{lemma}
\begin{proof}
Let $\{\sigma_m\}_{m\geq1}$ be a decreasing sequence of positive numbers that converges to $0$. Then for any $n\geq1$, $n^{-\sigma_m}\to1$ as $m\to\infty$. In view of \eqref{converge}, Lemmas \ref{decreasing} and \ref{con-sigma}, for each $m\geq1$, $f_{\sigma_m}(s)=T_{\sigma_m}f(s)=\sum_{n=1}^{\infty}a_nn^{-\sigma_m}X_nn^{-s}$ belongs to $\h^{p,q}_{\alpha}$ almost surely, and $f_{\sigma_{m}}\to f$ in $\h^{p,q}_{\alpha}$ almost surely as $m\to\infty$. Therefore, we can apply the Marcinkiewicz--Zygmund--Kahane theorem (see \cite[p. 240, Theorem II.4]{LQ2} and \cite[p. 11]{CFL22}) to deduce that $S_Nf\to f$ in $\h^{p,q}_{\alpha}$ almost surely as $N\to\infty$.
\end{proof}

The following result generalizes \cite[Theorem 8]{CFL22} to the Dirichlet series setting.

\begin{theorem}
 Assume $0<p<\infty$, $0<q\leq\infty$ and  $\alpha>-1.$ Let  $\{X_n\}$ be a standard random sequence. Let $f(s)=\sum_{n=1}^{\infty}a_nn^{-s}$ belong to $\D$. Then the following statements are equivalent:
\begin{enumerate}
	\item [(a)] $\R f\in\h^{p,q}_{\alpha}$ almost surely;
	\item [(b)] $\sup_{N\geq1}\left\|\sum_{n=1}^Na_nX_nn^{-s}\right\|_{\h^{p,q}_{\alpha}}<\infty$ almost surely;
	\item [(c)] $\E\left(\|\R f\|^t_{\h^{p,q}_{\alpha}}\right)<\infty$ for some $t>0$;
	\item [(d)] $\E\left(\|\R f\|^t_{\h^{p,q}_{\alpha}}\right)<\infty$ for any $t>0$;
	\item [(e)] $\E\left(\exp\left(\lambda\|\R f\|^r_{\h^{p,q}_{\alpha}}\right)\right)<\infty$ for some $\lambda>0$, where $r=\min\{p,q,1\}$;
	\item [(f)] $f\in\h^{2,q}_{\alpha}$.
\end{enumerate}
\end{theorem}
\begin{proof}
The implication (a)$\Longrightarrow$(b) follows from Lemma \ref{random-partial}, and the implication (b)$\Longrightarrow$(e) follows from  \cite[Lemma 9]{CFL22}. The implications (e)$\Longrightarrow$(d)$\Longrightarrow$(c)$\Longrightarrow$(a) are clear. It remains to establish the equivalence (d)$\Longleftrightarrow$(f). In the case $q<\infty$, for any $t>0$, by Fubini's theorem and the Khintchine--Kahane inequality (see \cite[p. 139, Theorem V.1]{LQ1} and \cite[Lemma 11]{CFL22}), we establish that
\begin{align*}
&\bigg(\E\Big(\|\R f\|^t_{\h^{p,q}_{\alpha}}\Big)\bigg)^{1/t}\\
&\quad\asymp\left(\E\left(\|\R f\|^q_{\h^{p,q}_{\alpha}}\right)\right)^{1/q}\\
&\quad=\left(\E\int_0^{+\infty}\left\|(\R f)_{\sigma}\right\|^q_{\h^p}d\mu_{\alpha}(\sigma)\right)^{1/q}\\
&\quad=\left(\int_0^{+\infty}\E\big(\left\|(\R f)_{\sigma}\right\|^q_{\h^p}\big)d\mu_{\alpha}(\sigma)\right)^{1/q}\\
&\quad\asymp\left(\int_0^{+\infty}\Big(\E\big(\left\|(\R f)_{\sigma}\right\|^p_{\h^p}\big)\Big)^{q/p}d\mu_{\alpha}(\sigma)\right)^{1/q}\\
&\quad=\left(\int_0^{+\infty}\left(\E\int_{\TT}\left|\sum_{n=1}^{\infty}a_nn^{-\sigma}X_nz^{\nu(n)}\right|^p
  dm_{\infty}(z)\right)^{q/p}d\mu_{\alpha}(\sigma)\right)^{1/q}\\
&\quad=\left(\int_0^{+\infty}\left(\int_{\TT}\E\left|\sum_{n=1}^{\infty}a_nn^{-\sigma}X_nz^{\nu(n)}\right|^p
  dm_{\infty}(z)\right)^{q/p}d\mu_{\alpha}(\sigma)\right)^{1/q}\\
&\quad\asymp\left(\int_{0}^{+\infty}\left(\sum_{n=1}^{\infty}|a_n|^2n^{-2\sigma}\right)^{q/2}d\mu_{\alpha}(\sigma)\right)^{1/q},
\end{align*}
which, in conjunction with \eqref{222}, yields the equivalence (d)$\Longleftrightarrow$(f) for this case. In the case $q=\infty$, it is similar to obtain that
$$\bigg(\E\Big(\|\R f\|^t_{\h^p}\Big)\bigg)^{1/t}\asymp\left(\sum_{n=1}^{\infty}|a_n|^2\right)^{1/2},$$
which also implies the equivalence (d)$\Longleftrightarrow$(f) and finishes the proof.
\end{proof}

In particular, we can determine the symbol spaces  $(\h^{p,q}_{\alpha})_{\star}$ for all $0<p<\infty$, $0<q\leq\infty$ and $\alpha>-1$ below,  which covers Theorem \ref{symbolA}. Consequently, Theorem \ref{remHardy} holds automatically as explained in the Introduction.

\begin{corollary}\label{sym}
Let $0<p<\infty$, $0<q\leq\infty$, $\alpha>-1$, and let $\{X_n\}$ be a standard random sequence. Then $(\h^{p,q}_{\alpha})_{\star}=\h^{2,q}_{\alpha}$. In particular, $(\h^p)_{\star}=\h^2$ and $(\A^p_{\alpha})_{\star}=\h^{2,p}_{\alpha}$.
\end{corollary}

\section{Embedding theorems}\label{sec4}

In this section, we are going to prove Theorem \ref{em} and  establish the Littlewood-type theorem for the Bergman spaces of Dirichlet series.  To this end, we establish two auxiliary results.  The first one estimates the growth of $\|f_{\sigma}\|_{\h^p}$ for $f\in\h^{p,q}_{\alpha}$, and the second one gives a family of test functions in the spaces $\h^{p,q}_{\alpha}$.

\begin{lemma}\label{mean-growth}
Let $0<p,q<\infty$ and $\alpha>-1$. Then for any $f\in\h^{p,q}_{\alpha}$, $\sigma>0$ and $\kappa\in(0,1]$,
$$\|f_{\sigma}\|_{\h^p}\leq\left(\frac{\Gamma(\alpha+2)}{2^{\alpha+1}}\right)^{1/q}\kappa^{-\frac{\alpha+1}{q}}\sigma^{-\frac{\alpha+1}{q}}
e^{\frac{2\kappa\sigma}{q}}\|f\|_{\h^{p,q}_{\alpha}}.$$
\end{lemma}
\begin{proof}
In view of Theorem \ref{sigma} and Lemma \ref{decreasing},  the function $\delta\mapsto\|f_{\delta}\|_{\h^p}$ is decreasing on $(0,+\infty)$. Therefore,
\begin{align*}
\|f\|^q_{\h^{p,q}_{\alpha}}
&=\int_0^{+\infty}\|f_{\delta}\|^q_{\h^p}d\mu_{\alpha}(\delta)\\
&\geq\frac{2^{\alpha+1}}{\Gamma(\alpha+1)}\int_0^{\sigma}\|f_{\delta}\|^q_{\h^p}\delta^{\alpha}e^{-2\delta}d\delta\\
&\geq\frac{2^{\alpha+1}}{\Gamma(\alpha+1)}\|f_{\sigma}\|^q_{\h^p}\int_0^{\kappa\sigma}\delta^{\alpha}e^{-2\delta}d\delta\\
&\geq\frac{2^{\alpha+1}}{\Gamma(\alpha+2)}(\kappa\sigma)^{\alpha+1}e^{-2\kappa\sigma}\|f_{\sigma}\|^q_{\h^p},
\end{align*}
which completes the proof.
\end{proof}

\begin{lemma}\label{lacunary}
Let $0<p,q<\infty$ and $\alpha>-1$. Suppose that $f(s)=\sum_{n=0}^{\infty}a_n2^{-2^ns}$ belongs to $\D$. Then $f\in\h^{p,q}_{\alpha}$ if and only if $\left\{2^{-(\alpha+1)n/q}a_n\right\}\in l^q$.
\end{lemma}
\begin{proof}
Suppose first that $\left\{2^{-(\alpha+1)n/q}a_n\right\}\in l^q$. Then for any $\sigma>0$, $$\sum_{n=0}^{\infty}2^{-2^{n+1}\sigma}|a_n|^2<\infty,$$
which implies that
$$\mathscr{B}(f_{\sigma})(\xi)=\sum_{n=0}^{\infty}a_n2^{-2^n\sigma}\xi^{2^n},\quad \xi\in\DD$$
belongs to the Hardy space $H^2(\DD)$. Applying \cite[Theorem 6.2.2]{JVA}, $\mathscr{B}(f_{\sigma})$ belongs to the Hardy space $H^p(\DD)$, which in turn implies that $f_{\sigma}\in\h^p$ and $\|f_{\sigma}\|_{\h^p}\asymp\|f_{\sigma}\|_{\h^2}$. Therefore, the change of variables $r=2^{-\sigma}$ gives that
\begin{align*}
\int_0^{+\infty}\|f_{\sigma}\|^q_{\h^{p}}d\mu_{\alpha}(\sigma)
&\asymp\int_{0}^{+\infty}\|f_{\sigma}\|^q_{\h^2}\sigma^{\alpha}e^{-2\sigma}d\sigma\\
&=\int_0^{+\infty}\left(\sum_{n=0}^{\infty}|a_n|^22^{-2^{n+1}\sigma}\right)^{q/2}\sigma^{\alpha}e^{-2\sigma}d\sigma\\
&\asymp\int_0^1\left(\sum_{n=0}^{\infty}|a_n|^2r^{2^{n+1}}\right)^{q/2}r^{\frac{2}{\log2}-1}\log^{\alpha}\frac{1}{r}dr\\
&\asymp\int_0^1\left(\sum_{n=0}^{\infty}|a_n|^2r^{2^{n+1}}\right)^{q/2}(1-r)^{\alpha}dr.
\end{align*}
Applying \cite[Theorem 6]{MP83}, we establish that
\begin{equation}\label{lacu}
\int_0^{+\infty}\|f_{\sigma}\|^q_{\h^{p}}d\mu_{\alpha}(\sigma)\asymp\sum_{n=0}^{\infty}2^{-(\alpha+1)n}|a_n|^q.
\end{equation}
Since $\left\{2^{-(\alpha+1)n/q}a_n\right\}\in l^q$, in view of Theorem \ref{suff}, we obtain that $f\in\h^{p,q}_{\alpha}$.

Conversely, if $f\in\h^{p,q}_{\alpha}$,  Theorem \ref{sigma} yields that  $f_{\sigma}\in\h^p$ for all $\sigma>0$, and
$$\int_0^{+\infty}\|f_{\sigma}\|^q_{\h^p}d\mu_{\alpha}(\sigma)<\infty.$$
Hence \eqref{lacu} still holds, and consequently, $\left\{2^{-(\alpha+1)n/q}a_n\right\}\in l^q$.
\end{proof}

We are now ready to prove Theorem \ref{em}.

\begin{proof}[\bf{Proof of Theorem \ref{em}}]
We first consider the sufficient part. Suppose $p\geq u$ and $f\in\h^{p,q}_{\alpha}$. Then Theorem \ref{sigma} yields that $f_{\sigma}\in\h^p\subset\h^u$ for any $\sigma>0$. In view of Theorem \ref{suff}, it is sufficient to show that
$$\int_{0}^{+\infty}\|f_{\sigma}\|^v_{\h^u}d\mu_{\beta}(\sigma)<\infty$$
under the assumptions (i) or (ii).

If (i) holds, then we choose $0<\kappa<\min\left\{1,q/v\right\}$ and make use of  Lemma \ref{mean-growth} to obtain that
\begin{align*}
\int_{0}^{+\infty}\|f_{\sigma}\|^v_{\h^u}d\mu_{\beta}(\sigma)
&\lesssim\int_{0}^{+\infty}\|f_{\sigma}\|^v_{\h^p}\sigma^{\beta}e^{-2\sigma}d\sigma\\
&\lesssim\kappa^{-\frac{\alpha+1}{q}v}\|f\|^v_{\h^{p,q}_{\alpha}}\int_{0}^{+\infty}\sigma^{\beta-\frac{\alpha+1}{q}v}
e^{-2\sigma(1-\frac{\kappa v}{q})}d\sigma.
\end{align*}
Since $\beta-\frac{\alpha+1}{q}v>-1$ and $1-\frac{\kappa v}{q}>0$, we deduce that $\int_{0}^{+\infty}\sigma^{\beta-\frac{\alpha+1}{q}v}
e^{-2\sigma(1-\frac{\kappa v}{q})}d\sigma<\infty$, which implies $f\in\h^{u,v}_{\beta}$.

If (ii) holds, then we choose $\kappa\in(0,1)$ small enough so that $\gamma:=1-\kappa(v/q-1)>0$. Noting that $\gamma\leq1$, Lemma \ref{mean-growth} together with Lemma \ref{decreasing} yields that
\begin{align*}
&\int_{0}^{+\infty}\|f_{\sigma}\|^v_{\h^u}d\mu_{\beta}(\sigma)\\
&\quad\lesssim\int_{0}^{+\infty}\|f_{\sigma}\|^{v-q}_{\h^p}\|f_{\sigma}\|^{q}_{\h^p}\sigma^{\beta}e^{-2\sigma}d\sigma\\
&\quad\lesssim\kappa^{-\frac{\alpha+1}{q}(v-q)}\|f\|^{v-q}_{\h^{p,q}_{\alpha}}
    \int_0^{+\infty}\|f_{\sigma}\|^q_{\h^p}\sigma^{\beta-\frac{\alpha+1}{q}(v-q)}e^{-2\sigma+\frac{2\kappa\sigma}{q}(v-q)}d\sigma\\
&\quad=\kappa^{-\frac{\alpha+1}{q}(v-q)}\|f\|^{v-q}_{\h^{p,q}_{\alpha}}
    \int_0^{+\infty}\|f_{\sigma}\|^q_{\h^p}\sigma^{\alpha}e^{-2\gamma\sigma}d\sigma\\
&\quad=\gamma^{-(\alpha+1)}\kappa^{-\frac{\alpha+1}{q}(v-q)}\|f\|^{v-q}_{\h^{p,q}_{\alpha}}
    \int_0^{+\infty}\|f_{\frac{\sigma}{\gamma}}\|^q_{\h^p}\sigma^{\alpha}e^{-2\sigma}d\sigma\\
&\quad\leq\gamma^{-(\alpha+1)}\kappa^{-\frac{\alpha+1}{q}(v-q)}\|f\|^{v-q}_{\h^{p,q}_{\alpha}}
    \int_0^{+\infty}\|f_{\sigma}\|^q_{\h^p}\sigma^{\alpha}e^{-2\sigma}d\sigma\\
&\quad\lesssim\gamma^{-(\alpha+1)}\kappa^{-\frac{\alpha+1}{q}(v-q)}\|f\|^{v}_{\h^{p,q}_{\alpha}}.
\end{align*}
Therefore, $f\in\h^{u,v}_{\beta}$ as well.

We next establish the necessary part. Assume neither (i) nor (ii) holds. We are going to construct some $f\in\D$ so that $f\in\h^{p,q}_{\alpha}$ but $f\notin\h^{u,v}_{\beta}$.

In the case that $p\geq u$ and $\frac{\alpha+1}{q}>\frac{\beta+1}{v}$, we consider the function $f_1(s)=\sum_{n=0}^{\infty}a_n2^{-2^ns}$ with $a_n=2^{(\beta+1)n/v}$. Since
$$\left\{2^{-\frac{(\alpha+1)n}{q}}a_n\right\}=\left\{2^{-\left(\frac{\alpha+1}{q}-\frac{\beta+1}{v}\right)n}\right\}\in l^q
\quad\text{and}\quad
\left\{2^{-\frac{(\beta+1)n}{v}}a_n\right\}=\left\{1\right\}\notin l^v,$$
it follows from Lemma \ref{lacunary} that $f_1\in\h^{p,q}_{\alpha}\setminus\h^{u,v}_{\beta}$.

In the case that $p\geq u$, $\frac{\alpha+1}{q}=\frac{\beta+1}{v}$ and $q>v$, we consider the function $f_2(s)=\sum_{n=0}^{\infty}a_n2^{-2^ns}$ with $a_n=2^{(\alpha+1)n/q}n^{-1/v}$. Since
$$\left\{2^{-\frac{(\alpha+1)n}{q}}a_n\right\}=\left\{n^{-\frac{1}{v}}\right\}\in l^q
\quad\text{and}\quad
\left\{2^{-\frac{(\beta+1)n}{v}}a_n\right\}=\left\{n^{-\frac{1}{v}}\right\}\notin l^v,$$
it follows again from Lemma \ref{lacunary} that $f_2\in\h^{p,q}_{\alpha}\setminus\h^{u,v}_{\beta}$.

Suppose finally that $p<u$. Fix $1/u<\eta<1/p$ and write
$k=\left[\frac{u(\beta+1)}{(u\eta-1)v}\right]+1.$
Define
$$f_3(s)=\prod_{j=1}^k\frac{1}{(1-p_j^{-s})^{\eta}},\quad s\in\C_0.$$
Since $p\eta<1$, we know that $f_3\in\h^p\subset\h^{p,q}_{\alpha}$. For any $\sigma>0$, the Forelli--Rudin estimate (see \cite[Proposition 1.4.10]{Ru80}) yields that
\begin{align*}
\|(f_3)_{\sigma}\|^u_{\h^u}=\|\mathscr{B}( (f_3)_{\sigma} )\|^u_{H^u(\TT)}&=\prod_{j=1}^k\int_{\T}\frac{1}{|1-p^{-\sigma}_jz_j|^{u\eta}}dm(z_j)\\
&\asymp\prod_{j=1}^k(1-p_j^{-2\sigma})^{1-u\eta}\\
&\geq(1-p_k^{-2\sigma})^{k(1-u\eta)},
\end{align*}
which implies that
\begin{align*}
\int_0^{+\infty}\|(f_3)_{\sigma}\|^v_{\h^u}d\mu_{\beta}(\sigma)
&\gtrsim\int_0^1\|(f_3)_{\sigma}\|^v_{\h^u}\sigma^{\beta}d\sigma\\
&\gtrsim\int_0^1(1-p_k^{-2\sigma})^{k(v/u-v\eta)}\sigma^{\beta}d\sigma\\
&\asymp\int_0^1\sigma^{\beta+kv(1/u-\eta)}d\sigma=\infty
\end{align*}
since $\beta+kv(1/u-\eta)<-1$. Therefore, $f_3\notin \h^{u,v}_{\beta}$. The proof is complete.
\end{proof}

Combining Theorem \ref{em} with Corollary \ref{sym}, we obtain the following Littlewood-type theorem for mixed norm spaces of Dirichlet series, which covers Theorem \ref{r-embedding}.

\begin{corollary}
Let $0<p,q,u,v<\infty$, $\alpha,\beta>-1$, and let $\{X_n\}$ be a standard random sequence. Then, $\R f\in\h^{u,v}_{\beta}$ almost surely for each $f\in\h^{p,q}_{\alpha}$ if and only if
\begin{enumerate}
	\item [(i)] $p\geq2$ and $\frac{\alpha+1}{q}<\frac{\beta+1}{v}$; or
	\item [(ii)] $p\geq2$, $\frac{\alpha+1}{q}=\frac{\beta+1}{v}$ and $q\leq v$.
\end{enumerate}
\end{corollary}

It is also interesting to consider the random embedding between Hardy and Bergman spaces of Dirichlet series. To this end, we need the following description for the inclusion between Hardy spaces and mixed norm spaces of Dirichlet series.

\begin{proposition}\label{hardy}
Let $0<p,u,v<\infty$ and $\alpha>-1$. Then
\begin{enumerate}
	\item $\h^p\subset\h^{u,v}_{\alpha}$ if and only if $p\geq u$;
	\item $\h^{u,v}_{\alpha}\nsubseteq\h^p$.
\end{enumerate}
\end{proposition}
\begin{proof}
(1) If $p\geq u$, then  Lemma \ref{con-sigma} gives that $\h^p\subset\h^u\subset\h^{u,v}_{\alpha}$. If $p<u$, then the proof of Theorem \ref{em} indicates that there exists $f\in\h^p\setminus\h^{u,v}_{\alpha}$.

(2) Consider the function $f(s)=\sum_{n=0}^{\infty}2^{-2^ns},$ then  Lemma \ref{lacunary} yields that $f\in\h^{u,v}_{\alpha}$. However, it is clear that $f\notin\h^2$, which together with \cite[Theorem 6.2.2]{JVA} implies $f\notin\h^p$.
\end{proof}

Combining Proposition \ref{hardy} with Corollary \ref{sym}, we can obtain the following Corollary \ref{Rmix0}, which covers Theorem \ref{Rmix}. Furthermore,   the answer to $\R:\A^p_{\alpha}\hookrightarrow\h^q$ is trivial since  $\A^p_{\alpha}\setminus\h^2\neq\emptyset$ for any $0<p<\infty$ due to Proposition \ref{hardy}.

\begin{corollary}\label{Rmix0}
 Assume $0<p,u,v<\infty$ and  $\alpha>-1.$ Let $\{X_n\}$ be a standard random sequence. Then, $\R f\in\h^{u,v}_{\alpha}$ almost surely for each $f\in\h^p$ if and only if $p\geq2$.
\end{corollary}

\section{Superposition operators}\label{super}

In this section, we are going to characterize the superposition operators between mixed norm spaces of Dirichlet series. To this end, we need the following lemma, which gives a growth estimate for the Dirichlet series in $\h^{p,q}_{\alpha}$ that are supported on a single prime.

\begin{lemma}\label{grow}
Let $0<p<\infty$, $0<q\leq\infty$ and $\alpha>-1$. Suppose that $f(s)=\sum_{n=0}^{\infty}a_n2^{-ns}$ is an element in $\h^{p,q}_{\alpha}$. Then $f$ is analytic on $\C_0$, and for any $s\in\C_0$ with $\Re s<1$,
$$|f(s)|\lesssim\left(\Re s\right)^{-\frac{1}{p}-\frac{\alpha+1}{q}}\|f\|_{\h^{p,q}_{\alpha}}.$$
\end{lemma}
\begin{proof}
Since $f\in\h^{p,q}_{\alpha}$, Theorem \ref{sigma} (or Lemma \ref{decreasing}) implies that $f_{\sigma}\in\h^p$ for any $\sigma>0$, which is equivalent to $\mathscr{B}(f_{\sigma})$ belongs to the Hardy space $H^p(\DD)$. Hence $\mathscr{B}(f_{\sigma})$ is analytic on the unit disk $\DD$, and consequently, $f_{\sigma}(s)=\mathscr{B}(f_{\sigma})(2^{-s})$ is analytic on $\C_0$. Since $\sigma>0$ is arbitrary, we obtain that $f$ is analytic on $\C_0$.

Fix $s\in\C_0$ with $\Re s<1$. Then for any $0<\sigma<\Re s/2$, \cite[Theorem 9.1]{Zhu07} yields that
\begin{align*}
|f(s)|=|f_{\sigma}(s-\sigma)|
&=\left|\mathscr{B}(f_{\sigma})\left(2^{-(s-\sigma)}\right)\right|\\
&\leq\frac{1}{\left(1-2^{-(2\Re s-2\sigma)}\right)^{1/p}}\|\mathscr{B}(f_{\sigma})\|_{H^p(\DD)}\\
&\leq\frac{1}{\left(1-2^{-\Re s}\right)^{1/p}}\|f_{\sigma}\|_{\h^p}\\
&\asymp(\Re s)^{-1/p}\|f_{\sigma}\|_{\h^p}.
\end{align*}
In the case $q=\infty$, the estimate $|f(s)|\lesssim(\Re s)^{-1/p}\|f\|_{\h^p}$ follows directly. In the case $q<\infty$, taking the power of $q$, integrating with respect to $d\mu_{\alpha}(\sigma)$ on $[0,\Re s/2]$ and using Theorem \ref{sigma}, we can obtain the desired estimate.
\end{proof}

The following lemma establishes a necessary condition for the superposition operators $S_{\varphi}$ mapping $\h^{p,q}_{\alpha}$ into $\h^{u,v}_{\beta}$.

\begin{lemma}\label{poly}
Let $0<p,u<\infty$, $0<q,v\leq\infty$ and $\alpha,\beta>-1$. Suppose that $\varphi$ is a function on $\C$ such that the superposition operator $S_{\varphi}$ maps $\h^{p,q}_{\alpha}$ into $\h^{u,v}_{\beta}$. Then $\varphi$ is a polynomial.
\end{lemma}
\begin{proof}
 Applying the same method as in the proof of \cite[Theorem 9]{BCGMS}, we can obtain that $\varphi$ is an entire function (see also \cite[Lemma 5.1]{FGMS}). Suppose now that $q,v<\infty$ and we will show that $\varphi$ is a polynomial of degree at most
$$N=\left[\frac{pq(u(\beta+1)+v)}{uv(p(\alpha+1)+q)}\right].$$
To this end, it suffices to show that
\begin{equation}\label{suffice}
\lim_{r\to\infty}\frac{M_{\infty}(\varphi,r)}{r^{N+1}}=0,
\end{equation}
where $M_{\infty}(\varphi,r)=\sup_{|\xi|=r}|\varphi(\xi)|$. We complete the proof by contradiction. Assume that \eqref{suffice} does not hold. Then there exist $\kappa>0$ and a sequence $\{\xi_n\}\subset\C$ with $|\xi_n|>1$ such that $|\xi_n|\to\infty$ and
\begin{equation}\label{big}
|\varphi(\xi_n)|\geq\kappa|\xi_n|^{N+1},\quad \forall \ n\geq1.
\end{equation}

Fix now
$$\max\left\{\frac{1}{N+1}\left(\frac{\beta+1}{v}+\frac{1}{u}\right),\frac{1}{p}\right\}<\gamma<\frac{\alpha+1}{q}+\frac{1}{p}$$
and define
$$g(s)=(1-2^{-s})^{-\gamma},\quad s\in\C_0.$$
Then in view of  \cite[Proposition 1.4.10]{Ru80}, for any $\sigma>0$,
$$\|g_{\sigma}\|^p_{\h^p}=\|\mathscr{B}(g_{\sigma})\|^p_{H^p(\DD)}=\int_{\T}\frac{dm(\xi)}{|1-2^{-\sigma}\xi|^{p\gamma}}\asymp
\left(1-2^{-\sigma}\right)^{1-p\gamma}.$$
Consequently,
$$\int_0^{+\infty}\|g_{\sigma}\|^q_{\h^p}d\mu_{\alpha}(\sigma)
\lesssim\int_0^1\sigma^{\alpha+\frac{q}{p}-q\gamma}d\sigma+\int_1^{+\infty}\sigma^{\alpha}e^{-2\sigma}d\sigma<\infty,$$
due to $\alpha+\frac{q}{p}-q\gamma>-1$. Hence Theorem \ref{suff} yields that $g\in\h^{p,q}_{\alpha}$.

Using the same argument as in the proof of \cite[Proposition 3.1]{BV}, we may assume that $|\arg \xi_n|<\frac{\gamma\pi}{4}$. For each $n\geq1$, let
$$s_n=-\log_2\left(1-\xi_n^{-1/\gamma}\right),$$
where the complex logarithm is chosen so that $0\leq\arg s_n<\frac{2\pi}{\log2}$. Since $|\xi_n|>1$ and $|\arg \xi_n|<\frac{\gamma\pi}{4}$ for each $n\geq1$, all the points $2^{-s_n}=1-\xi_n^{-1/\gamma}$ belong to a Stolz domain with vertex at $1$. Hence we can find some $c>0$ such that
\begin{equation}\label{Sto}
|1-2^{-s_n}|\leq c(1-2^{-\Re s_n}),\quad\forall \ n  \geq1.
\end{equation}

Since $g\in\h^{p,q}_{\alpha}$, $\varphi\circ g=S_{\varphi}g$ belongs to $\h^{u,v}_{\beta}$. Then  Lemma \ref{grow} together with \eqref{big} and \eqref{Sto} yields that
\begin{align*}
\kappa|\xi_n|^{N+1}\leq|\varphi(\xi_n)|&=|\varphi\circ g(s_n)|\\
&\lesssim(\Re s_n)^{-\frac{1}{u}-\frac{\beta+1}{v}}\\
&\asymp\left(1-2^{-\Re s_n}\right)^{-\frac{1}{u}-\frac{\beta+1}{v}}\\
&\lesssim|1-2^{-s_n}|^{-\frac{1}{u}-\frac{\beta+1}{v}}\\
&=|\xi_n|^{\left(\frac{1}{u}+\frac{\beta+1}{v}\right)/\gamma},
\end{align*}
which is a contradiction since $\frac{1}{u}+\frac{\beta+1}{v}<(N+1)\gamma$ and $|\xi_n|\to\infty$. Therefore, \eqref{suffice} holds and $\varphi$ is a polynomial of degree at most $N$. In the case $q=\infty$ or $v=\infty$, one can obtain the desired result by minor modifications of the above argument. We omit the details.
\end{proof}

For an integer $d\geq1$, we use $A_d$ to denote what Bohr called {\it der $d$te Abschnitt}, defined by
$$A_d\left(\sum_{n=1}^{\infty}a_nn^{-s}\right):=\sum_{n:P^+(n)\leq p_d}a_nn^{-s},$$
where $P^+(n)$ is the largest prime divisor of $n$. The following characterization for $\h^p$ was established in \cite[Theorem 2.1]{BBSS} and \cite[Proposition 2.3]{DG}.

\begin{lemma}\label{d-ab}
Let $0<p<\infty$ and $f\in\D$. Then $f\in\h^p$ if and only if
$$\sup_{d\geq1}\|A_df\|_{\h^p}<\infty.$$
Moreover, $\|f\|_{\h^p}=\sup_{d\geq1}\|A_df\|_{\h^p}.$
\end{lemma}

Based on Lemma \ref{d-ab}, we have the following fundamental result.

\begin{proposition}\label{NN}
Let $0<p<\infty$, $f\in\D$, and let $N$ be a positive integer. Then $f^N\in\h^{p/N}$ if and only if $f\in\h^{p}$. Moreover,
$$\|f^N\|_{\h^{p/N}}=\|f\|^N_{\h^p}.$$
\end{proposition}
\begin{proof}
If $f\in\h^p$, then it follows from \cite[Theorem 9]{BCGMS} that $f^N\in\h^{p/N}$ with $\|f^N\|_{\h^{p/N}}\leq\|f\|^N_{\h^p}$. Assume now that $f^N\in\h^{p/N}$. Then by Lemma \ref{d-ab}, $(A_df)^N=A_d(f^N)\in\h^{p/N}$ for any positive integer $d$, and $\|(A_df)^N\|_{\h^{p/N}}\leq\|f^N\|_{\h^{p/N}}$. Since the Dirichlet series $(A_df)^N$ depends only on the first $d$ primes, its Bohr lift $\mathscr{B}\big((A_df)^N\big)$ belongs to the Hardy space $H^{p/N}(\DD^d)$ over the polydisk $\DD^d$, which implies that the power series  $\mathscr{B}\big((A_df)^N\big)$ converges absolutely and uniformly on each compact subset of $\DD^d$. Consequently, $$(A_df)^N(s)=\mathscr{B}\big((A_df)^N\big)(2^{-s},3^{-s},\cdots,p_d^{-s})$$
converges absolutely and uniformly on $\C_{\epsilon}$ for any $\epsilon>0$. Therefore, $A_df$ converges absolutely and uniformly on $\C_{\epsilon}$ for any $\epsilon>0$, and then $\mathscr{B}(A_df)$ is analytic on $\DD^d$. Comparing the coefficients, we deduce that
$$\big(\mathscr{B}(A_df)\big)^N=\mathscr{B}\big((A_df)^N\big)\in H^{p/N}(\DD^d).$$
Therefore, $\mathscr{B}(A_df)\in H^p(\DD^d)$, and
\begin{align*}
\left\|\mathscr{B}(A_df)\right\|_{H^p(\DD^d)}&=\left\|\big(\mathscr{B}(A_df)\big)^N\right\|^{1/N}_{H^{p/N}(\DD^d)}
=\left\|\mathscr{B}\big((A_df)^N\big)\right\|^{1/N}_{H^{p/N}(\DD^d)}\\
&=\left\|(A_df)^N\right\|^{1/N}_{\h^{p/N}}\leq\|f^N\|^{1/N}_{\h^{p/N}}.
\end{align*}
Hence $A_df\in\h^{p}$ and $\|A_df\|_{\h^{p}}\leq\|f^N\|^{1/N}_{\h^{p/N}}$. Since $d\geq1$ is arbitrary, using Lemma \ref{d-ab} again, we obtain that $f\in\h^p$, and $\|f\|^N_{\h^p}\leq\|f^N\|_{\h^{p/N}}$.
\end{proof}

As a direct consequence of Proposition \ref{NN}, Theorems \ref{sigma} and \ref{suff}, we have the following result concerning powers of Dirichlet series in the spaces $\h^{p,q}_{\alpha}$.

\begin{corollary}\label{N-power}
Let $0<p,q<\infty$, $\alpha>0$, and let $N$ be a positive integer. Then for any $f\in\D$, $f\in\h^{p,q}_{\alpha}$ if and only if $f^N\in\h^{\frac{p}{N},\frac{q}{N}}_{\alpha}$. Moreover,
$$\|f^N\|_{\h^{\frac{p}{N},\frac{q}{N}}_{\alpha}}=\|f\|^N_{\h^{p,q}_{\alpha}}.$$
\end{corollary}

We can now characterize the monomials $\varphi$ that induce superposition operators $S_{\varphi}$ mapping $\h^{p,q}_{\alpha}$ into $\h^{u,v}_{\beta}$.

\begin{proposition}\label{mo}
Let $0<p,q,u,v<\infty$ and $\alpha,\beta>-1$. Suppose that $N$ is a non-negative integer and $\varphi(z)=z^N$. Then $S_{\varphi}$ maps $\h^{p,q}_{\alpha}$ into $\h^{u,v}_{\beta}$ if and only if
\begin{enumerate}[(i)]
	\item $N\leq\min\left\{\frac{p}{u},\frac{q(\beta+1)}{v(\alpha+1)}\right\}$ in the case $\alpha\leq\beta$;
	\item $N\leq\frac{p}{u}$ and $N<\frac{q(\beta+1)}{v(\alpha+1)}$ in the case $\alpha>\beta$.
\end{enumerate}
Moreover, if (i) or (ii) holds, then $S_{\varphi}:\h^{p,q}_{\alpha}\to\h^{u,v}_{\beta}$ is bounded.
\end{proposition}
\begin{proof}
If $N=0$, it is clear that $S_{\varphi}$ maps $\h^{p,q}_{\alpha}$ into $\h^{u,v}_{\beta}$. So we assume that $N\geq1$. Then it follows from Corollary \ref{N-power} that, for any $f\in\h^{p,q}_{\alpha}$, $S_{\varphi}f=f^N\in\h^{u,v}_{\beta}$ if and only if $f\in\h^{Nu,Nv}_{\beta}$. Therefore, $S_{\varphi}$ maps $\h^{p,q}_{\alpha}$ into $\h^{u,v}_{\beta}$ if and only if $\h^{p,q}_{\alpha}\subset\h^{Nu,Nv}_{\beta}$, which, due to Theorem \ref{em}, is equivalent to the fact that one of the following conditions holds:

($\dag$) $p\geq Nu$ and $\frac{\alpha+1}{q}<\frac{\beta+1}{Nv}$; or

($\ddag$) $p\geq Nu$, $\frac{\alpha+1}{q}=\frac{\beta+1}{Nv}$ and $q\leq Nv$.\\
It is trivial to verify that either ($\dag$) or ($\ddag$) holds if and only if $N$ satisfies (i) or (ii). Moreover, in this case, by the proof of Theorem \ref{em}, the inclusion $\h^{p,q}_{\alpha}\subseteq\h^{Nu,Nv}_{\beta}$ is bounded, so $S_{\varphi}:\h^{p,q}_{\alpha}\to\h^{u,v}_{\beta}$ is bounded, which completes the proof.
\end{proof}

The following lemma is a comparison principle for superposition operators, which indicates that if $\varphi$ is a polynomial such that $S_{\varphi}$ maps $\h^{p,q}_{\alpha}$ into $\h^{u,v}_{\beta}$, then for any polynomial $\psi$ whose degree is at most the degree of $\varphi$, $S_{\psi}$ also maps $\h^{p,q}_{\alpha}$ into $\h^{u,v}_{\beta}$.

\begin{lemma}\label{ppsi}
Let $0<p,u<\infty$, $0<q,v\leq\infty$, $\alpha,\beta>-1$, and let $\varphi$ be a polynomial of degree $k$. Suppose that the superposition operator $S_{\varphi}$ maps $\h^{p,q}_{\alpha}$ into $\h^{u,v}_{\beta}$. Then for any polynomial $\psi$ of degree $j$ with $0\leq j\leq k$, $S_{\psi}$ maps $\h^{p,q}_{\alpha}$ into $\h^{u,v}_{\beta}$. Moreover, there exists $C>0$ such that for any $f\in\h^{p,q}_{\alpha}$,
$$\|S_{\psi}f\|_{\h^{u,v}_{\beta}}\leq C\left(\|S_{\varphi}f\|_{\h^{u,v}_{\beta}}+1\right).$$
\end{lemma}
\begin{proof}
We only consider the case $v<\infty$. The case $v=\infty$ is similar.
Since $j\leq k$, there exist $C_1,R>0$ such that
\begin{equation}\label{compare}
|\psi(\xi)|\leq C_1|\varphi(\xi)|\quad \text{if } |\xi|>R.
\end{equation}
Fix $f\in\h^{p,q}_{\alpha}$. Since $\varphi\circ f=S_{\varphi}f\in\h^{u,v}_{\beta}$, we deduce from Theorem \ref{sigma} that $\varphi\circ f_{\sigma}\in\h^u$ for any $\sigma>0$, and
$$\|S_{\varphi}f\|^v_{\h^{u,v}_{\beta}}=\int_0^{+\infty}\|\varphi\circ f_{\sigma}\|^v_{\h^u}d\mu_{\beta}(\sigma).$$
Then for any $d\geq1$, Lemma \ref{d-ab} yields that $\varphi\circ(A_df_{\sigma})=A_d(\varphi\circ f_{\sigma})\in\h^u$. Arguing as in the proof of Proposition \ref{NN}, we obtain that $\varphi\circ(A_df_{\sigma})$ converges uniformly on $\C_{\epsilon}$ for any $\epsilon>0$. Similarly, the fact that $f\in\h^{p,q}_{\alpha}$ implies that $A_df_{\sigma}$ converges uniformly on $\C_{\epsilon}$ for any $\epsilon>0$, and consequently, so does $\psi\circ(A_df_{\sigma})$. Hence $\mathscr{B}(A_df_{\sigma})$ and $\psi\circ\big(\mathscr{B}(A_df_{\sigma})\big)=\mathscr{B}\big(\psi\circ(A_df_{\sigma})\big)$ are both analytic on $\DD^d$. For any $0<r<1$, write
$$A_r:=\left\{(z_1,z_2,\cdots,z_d)\in\T^d:\left|\mathscr{B}(A_df_{\sigma})(rz_1,\cdots,rz_d)\right|>R\right\}.$$
Then by \eqref{compare},
\begin{align*}
&\int_{\T^d}\left|\psi\big(\mathscr{B}(A_df_{\sigma})(rz_1,\cdots,rz_d)\big)\right|^udm_d(z_1,\cdots,z_d)\\
&\quad\leq C_1^u\int_{A_r}\left|\varphi\big(\mathscr{B}(A_df_{\sigma})(rz_1,\cdots,rz_d)\big)\right|^udm_d(z_1,\cdots,z_d)+M_{\infty}(\psi,R)^u\\
&\quad\leq C_1^u\left\|\varphi\circ\big(\mathscr{B}(A_df_{\sigma})\big)\right\|^u_{H^u(\DD^d)}+M_{\infty}(\psi,R)^u\\
&\quad=C_1^u\|A_d(\varphi\circ f_{\sigma})\|^u_{\h^u}+M_{\infty}(\psi,R)^u\\
&\quad\leq C_1^u\|\varphi\circ f_{\sigma}\|^u_{\h^u}+M_{\infty}(\psi,R)^u.
\end{align*}
Since $r\in(0,1)$ is arbitrary, we obtain that $\psi\circ\big(\mathscr{B}(A_df_{\sigma})\big)\in H^u(\DD^d)$. Therefore, $A_d(\psi\circ f_{\sigma})\in\h^u$, and
$$\|A_d(\psi\circ f_{\sigma})\|^u_{\h^u}=\left\|\psi\circ\big(\mathscr{B}(A_df_{\sigma})\big)\right\|^u_{H^u(\DD^d)}
\leq C_1^u\|\varphi\circ f_{\sigma}\|^u_{\h^u}+M_{\infty}(\psi,R)^u.$$
In view of Lemma \ref{d-ab}, the arbitrariness of $d$ then yields $\psi\circ f_{\sigma}\in\h^u$, and
$$\|\psi\circ f_{\sigma}\|^u_{\h^u}
\leq C_1^u\|\varphi\circ f_{\sigma}\|^u_{\h^u}+M_{\infty}(\psi,R)^u.$$
Note that the above inequality holds for all $\sigma>0$. Consequently,
\begin{align*}
\int_0^{+\infty}\|\psi\circ f_{\sigma}\|^v_{\h^u}d\mu_{\beta}(\sigma)
&\lesssim \int_0^{+\infty}\|\varphi\circ f_{\sigma}\|^v_{\h^u}d\mu_{\beta}(\sigma)+M_{\infty}(\psi,R)^v\\
&=\|S_{\varphi}f\|^v_{\h^{u,v}_{\beta}}+M_{\infty}(\psi,R)^v.
\end{align*}
Hence by Theorem \ref{suff}, $S_{\psi}f=\psi\circ f$ belongs to $\h^{u,v}_{\beta}$, and the desired norm estimate holds. The arbitrariness of $f$ finishes the proof.
\end{proof}

We are now ready to display the main result of this section, which covers Theorem \ref{sBergman}.

\begin{theorem}\label{supe}
Let $0<p,q,u,v<\infty$, $\alpha,\beta>-1$, and let $\varphi$ be a function on $\C$. Then the superposition operator $S_{\varphi}$ maps $\h^{p,q}_{\alpha}$ into $\h^{u,v}_{\beta}$ if and only if $\varphi$ is a polynomial of degree $N$, where $N$ satisfies
\begin{enumerate}[(i)]
	\item $N\leq\min\left\{\frac{p}{u},\frac{q(\beta+1)}{v(\alpha+1)}\right\}$ in the case $\alpha\leq\beta$;
	\item $N\leq\frac{p}{u}$ and $N<\frac{q(\beta+1)}{v(\alpha+1)}$ in the case $\alpha>\beta$.
\end{enumerate}
Moreover, if $S_{\varphi}$ maps $\h^{p,q}_{\alpha}$ into $\h^{u,v}_{\beta}$, then it is bounded and continuous.
\end{theorem}
\begin{proof}
If $\varphi$ is a polynomial of degree $N$ that satisfies (i) or (ii), then in view of  Proposition \ref{mo}, $S_{\varphi}$ maps $\h^{p,q}_{\alpha}$ boundedly into $\h^{u,v}_{\beta}$. Applying Lemma \ref{growth}, we know that for any $s\in\C_{1/2}$, the point evaluation at $s$ is bounded on $\h^{u,v}_{\beta}$. Hence  $S_{\varphi}:\h^{p,q}_{\alpha}\to\h^{u,v}_{\beta}$ is continuous due to \cite[Theorem 3.1]{BR}. Conversely, if $S_{\varphi}$ maps $\h^{p,q}_{\alpha}$ into $\h^{u,v}_{\beta}$, then by Lemma \ref{poly}, $\varphi$ is a polynomial. Assume the degree of $\varphi$ is $N$. Then it follows from Lemma \ref{ppsi} that the function $\psi(\xi)=\xi^N$ also induces a superposition operators mapping $\h^{p,q}_{\alpha}$ into $\h^{u,v}_{\beta}$. Therefore,  it follows from Proposition \ref{mo} that  $N$ satisfies (i) or (ii). The proof is completed.
\end{proof}

We end this paper by characterizing the superposition operators between Hardy spaces and mixed norm spaces of Dirichlet series. The proof is similar with the one of Theorem \ref{supe} and so is omitted.

\begin{theorem}
Let $0<p,u,v<\infty$, $\alpha>-1$, and let $\varphi$ be a function on $\C$. Then the following assertions hold.
\begin{enumerate}
	\item $S_{\varphi}$ maps $\h^p$ into $\h^{u,v}_{\alpha}$ if and only if $\varphi$ is a polynomial of degree at most $p/u$.
	\item $S_{\varphi}$ maps $\h^{u,v}_{\alpha}$ into $\h^p$ if and only if $\varphi$ is constant.
\end{enumerate}
\end{theorem}

\begin{corollary}
Let $0<p,q<\infty$, $\alpha>-1$, and let $\varphi$ be a function on $\C$. Then the following assertions hold.
\begin{enumerate}
	\item $S_{\varphi}$ maps $\h^p$ into $\A^q_{\alpha}$ if and only if $\varphi$ is a polynomial of degree at most $p/q$.
	\item $S_{\varphi}$ maps $\A^p_{\alpha}$ into $\h^q$ if and only if $\varphi$ is constant.
\end{enumerate}
\end{corollary}


\end{document}